\documentclass[11pt]{article}
\usepackage[a4paper,margin=1.15in]{geometry}
\usepackage[T1]{fontenc}
\usepackage{amsmath,amssymb,amsthm,mathtools}
\usepackage{needspace}
\usepackage{xcolor}
\usepackage[colorlinks=true,linkcolor=blue,citecolor=blue,urlcolor=blue]{hyperref}
\numberwithin{equation}{section}
\newtheorem{theorem}{Theorem}[section]
\newtheorem{proposition}[theorem]{Proposition}
\newtheorem{corollary}[theorem]{Corollary}
\newtheorem{lemma}[theorem]{Lemma}
\newtheorem{remark}[theorem]{Remark}
\newtheorem{definition}[theorem]{Definition}
\newtheorem{example}[theorem]{Example}

\theoremstyle{remark}

\DeclareMathOperator{\supp}{supp}

\title{Worst-Case Relative Growth of Functionally Generated Portfolios under Restricted Market-Weight Dynamics}
\author{Jilong Xu\thanks{\raggedright School of Mathematics, Nanjing University, Nanjing 210093, P. R. China. Email: \href{mailto:xujilongnj@smail.nju.edu.cn}{\texttt{xujilongnj@smail.nju.edu.cn}}.}
\and Xiaojun Cui\thanks{\raggedright Corresponding author. School of Mathematics, Nanjing University, Nanjing 210093, P. R. China. Email: \href{mailto:xcui@nju.edu.cn}{\texttt{xcui@nju.edu.cn}}.}}
\date{Version v1}
\newcommand{\printfunding}{%
\begin{NoHyper}
\renewcommand{\thefootnote}{}
\footnotetext{\raggedright\textbf{Funding:} Xiaojun Cui and Jilong Xu are supported by the National Natural Science
Foundation of China (Grant No. 12171234), the Project Funded by the
Priority Academic Program Development of Jiangsu Higher Education
Institutions (PAPD) and the Fundamental Research Funds for the Central
Universities.}
\end{NoHyper}
}

\begin{document}
\maketitle
\printfunding
\begin{abstract}
We study the worst-case relative growth rate of a fixed functionally
generated portfolio under restrictions on market-weight dynamics over
a compact convex market state space. For a nonempty compact forward
shift-invariant path family, the pathwise wealth decomposition identifies
the global rate with the minimum invariant average of the one-step
\(L\)-divergence. For paths generated by a compact one-step transition
relation, this rate is also the minimum over balanced transition measures,
and a minimizing measure can be chosen on one relation basic set.

For a fixed initial state, we study the lower limit of the minimum
average \(L\)-divergence over the first \(t\) periods as \(t\to\infty\).
We represent this value using empirical transition laws generated from
that state. Their common state marginals are supported on chain-recurrent
states that are chain-accessible from the initial state. The lower bound
is obtained by minimizing the average \(L\)-divergence over balanced
transition measures on chain-accessible basic sets. Finite exact
reachability gives an upper bound. Examples show that both bounds can be
strict. We also give a finite exact reachability condition under which
the minimum average \(L\)-divergence over the first \(t\) periods converges
for paths confined to a basic set.
\end{abstract}



\section{Introduction}

Stochastic portfolio theory (SPT) studies portfolio performance relative to the market through the dynamics of market weights.  We work within its functional-generation framework \cite{Fernholz2002,FernholzKaratzas2009,KaratzasRuf2017,Wong2015Optimization} and study the following evaluation problem: for a \emph{fixed} functionally generated portfolio and a prescribed family of admissible market-weight paths, what is its worst-case relative growth rate?  We distinguish the global problem, in which the initial state is allowed to vary, from the problem with a fixed initial state.  A positive global rate guarantees that, for all sufficiently large horizons, the average logarithmic growth relative to the market is bounded below by a positive constant uniformly over the admissible paths.

This question is complementary to robust portfolio optimization under probabilistic model uncertainty.  Kardaras and Robertson optimize asymptotic growth over trading strategies against a class of probability laws with prescribed state space, instantaneous covariance, and limiting occupancy density \cite{KardarasRobertson2021}.  Itkin and Larsson impose long-only constraints and optimize over concave functionally generated portfolios within a related framework \cite{ItkinLarsson2022}.  Here, the portfolio is fixed in advance, and uncertainty is represented by the set of market paths declared admissible.  The formulation is deterministic: no probabilities are assigned to paths or to the possible successors of a state.  Continuous-time pathwise and rough-path formulations provide related motivation for dispensing with a fixed probabilistic law \cite{AllanCuchieroLiuPromel2023,KaratzasKim2020}.

Let \(K\subset\Delta_n\) be a nonempty compact convex restricted market state space in the open simplex, contained in a relative-open convex neighborhood \(U\subset\Delta_n\).  Let \(\Phi:U\to(0,\infty)\) be \(C^1\) and concave, set \(\phi=\log\Phi\), and let \(\pi\) be the associated functionally generated portfolio.  At the general level, admissible paths form a nonempty compact forward shift-invariant set \(\Omega\subset K^{\mathbb N_0}\), with left shift \(\sigma\).  The discrete pathwise wealth decomposition \cite[Lemma~7]{PalWong2016Geometry} gives
\begin{equation}\label{eq:intro-wealth-decomposition}
        \log V_\pi(t,\omega)
        =
        \phi(\omega_t)-\phi(\omega_0)
        +
        \sum_{k=0}^{t-1}T_\phi(\omega_{k+1}\mid\omega_k).
\end{equation}
Here \(T_\phi(q\mid p)\) is the continuous nonnegative one-step \(L\)-divergence associated with the transition from \(p\) to \(q\) \cite{PalWong2016Geometry,PalWong2018ExpConcave}.  The endpoint term is uniformly bounded on \(K\), so the global rate
\begin{equation}\label{eq:central-problem}
        \gamma(\Phi,\Omega)
        :=
        \lim_{t\to\infty}
        \min_{\omega\in\Omega}
        \frac1t\log V_\pi(t,\omega)
\end{equation}
is governed by the one-step cost
\begin{equation}\label{eq:intro-drift-function}
        f_\phi(\omega):=T_\phi(\omega_1\mid\omega_0).
\end{equation}
Theorem~\ref{thm:invariant-L-divergence-representation} shows that the limit exists and that
\begin{equation}\label{eq:intro-main-identity}
        \gamma(\Phi,\Omega)
        =
        \min_{\nu\in\mathcal M_\sigma(\Omega)}
        \int_\Omega f_\phi\,d\nu.
\end{equation}
Here \(\mathcal M_\sigma(\Omega)\) denotes the shift-invariant Borel probability measures on \(\Omega\).  This identity is the minimizing form of the standard ergodic-optimization relation between finite-horizon averages and invariant averages \cite[Proposition~2.1]{Jenkinson2006}.  Related invariant/time-average equivalences hold for upper-semicontinuous multivalued systems on compact metrizable spaces \cite[Theorem~3.3]{JenkinsonLiLiaoZhang2025}.

In the main setting, paths are generated by a compact one-step transition relation \(\Gamma\subset K\times K\). A path \(\omega\in K^{\mathbb N_0}\) belongs to \(\Omega_\Gamma\) if and only if \((\omega_k,\omega_{k+1})\in\Gamma\) for every \(k\ge0\). We assume \(\Omega_\Gamma\ne\varnothing\). A general restricted path space \(\Omega\) may impose additional constraints on whole paths. For \(\Omega_\Gamma\), invariant path measures correspond to balanced transition measures on \(\Gamma\), whose current-state and next-state marginals coincide. Write \(\mathcal I(\Gamma)\) for this class and \(\mathcal B(\Gamma)\) for the family of relation basic sets. For each \(B\in\mathcal B(\Gamma)\), let \(\Gamma_B:=\Gamma\cap(B\times B)\), and let \(\mathcal I(\Gamma_B)\) denote the balanced transition measures on \(\Gamma_B\). Define
\[
\beta_\phi(B):=
\min_{\lambda\in\mathcal I(\Gamma_B)}
\int_{\Gamma_B}T_\phi(q\mid p)\,d\lambda(p,q).
\]
The transition-measure representation (Theorem~\ref{thm:transition-measure-form}) and the basic-set minimum formula \eqref{eq:global-basic-set-minimum} give
\[
        \gamma(\Phi,\Omega_\Gamma)
        =
        \min_{\lambda\in\mathcal I(\Gamma)}
        \int_\Gamma T_\phi(q\mid p)\,d\lambda(p,q)
        =
        \min_{B\in\mathcal B(\Gamma)}\beta_\phi(B).
\]
In particular, a global minimizing balanced measure can be chosen on one basic set. The representation by invariant path measures and recurrent localization use the closed-relation framework of Miller and Akin \cite{MillerAkin1999}.

The global minimization in \eqref{eq:intro-main-identity} does not fix the initial state.  Fix \(p\in K_\Gamma^\infty\), so that at least one infinite admissible path starts from \(p\), and write \(\Omega_\Gamma(p)=\{\omega\in\Omega_\Gamma:\omega_0=p\}\).  Define
\[
        \mathcal V_t^\phi(p)
        :=
        \min_{\omega\in\Omega_\Gamma(p)}
        \sum_{k=0}^{t-1}T_\phi(\omega_{k+1}\mid\omega_k),
        \qquad
        \underline\alpha_\phi(p)
        :=
        \liminf_{t\to\infty}\frac{\mathcal V_t^\phi(p)}{t}.
\]
By the same endpoint bound, \(\underline\alpha_\phi(p)\) is also the lower limit, as \(t\to\infty\), of the minimum of \(t^{-1}\log V_\pi(t,\omega)\) over \(\omega\in\Omega_\Gamma(p)\).  The set \(\Omega_\Gamma(p)\) of admissible paths starting from \(p\) is generally not shift-invariant, and a finite-horizon minimizing path may depend on the horizon.  Thus, the global limit argument does not by itself give convergence from a fixed state or identify this lower rate with the infimum of long-run averages along single infinite paths.  The fixed-state formulation is related to the occupational-measure approach to discrete-time average-cost control \cite{BorkarGaitsgoryShvartsman2019}.

We represent the lower asymptotic rate \(\underline\alpha_\phi(p)\) through empirical transition laws generated from \(p\).  These are weak limits of
\[
        L_t^\omega
        :=
        \frac1t\sum_{k=0}^{t-1}\delta_{(\omega_k,\omega_{k+1})}
\]
along horizons \(t_j\to\infty\) and paths \(\omega^{(j)}\in\Omega_\Gamma(p)\), with the path allowed to vary with \(j\).  Writing \(\mathfrak A_\Gamma(p)\) for this class, we prove
\[
        \underline\alpha_\phi(p)
        =
        \min_{\lambda\in\mathfrak A_\Gamma(p)}
        \int_\Gamma T_\phi(y\mid x)\,d\lambda(x,y).
\]
Each such law is balanced, and its common state marginal is supported on recurrent states that are chain-accessible from \(p\). This gives a lower bound given by the minimum of \(\beta_\phi(B)\) over chain-accessible basic sets.

For a basic set \(B\) and \(q\in B\), let \(\Omega_{\Gamma_B}(q)\) denote the admissible paths starting from \(q\) and remaining in \(B\). The minimum accumulated \(L\)-divergence over the first \(t\) periods is
\[
\mathcal V_{t,B}^\phi(q)
:=\min_{\omega\in\Omega_{\Gamma_B}(q)}
\sum_{k=0}^{t-1}T_\phi(\omega_{k+1}\mid\omega_k).
\]
The corresponding lower asymptotic rate is
\[
\underline\alpha_{\phi,B}(q)
:=\liminf_{t\to\infty}\frac{\mathcal V_{t,B}^\phi(q)}{t}.
\]
For the complementary upper bound, if \(q\in B\) is finitely exactly reachable from \(p\), then
\[
\underline\alpha_\phi(p)\le\underline\alpha_{\phi,B}(q).
\]

Both bounds can be strict.
In particular, identifying the unique chain-accessible basic set
need not determine the empirical transition laws generated from the
current state; see Appendix~\ref{app:component-pressure-tests}.

The upper bound involves the restricted rate \(\underline\alpha_{\phi,B}(q)\). We seek conditions under which this equals \(\beta_\phi(B)\), allowing both bounds to be expressed using the same quantities. In fact, we give a sufficient condition for the stronger conclusion that \(\mathcal V_{t,B}^\phi(q)/t\) converges to \(\beta_\phi(B)\).
It suffices that \(q\) can reach a state \(x\in B\) through finitely
many exact \(\Gamma_B\)-transitions, and that an infinite admissible
path starting from \(x\) and remaining in \(B\) has an
\(L\)-divergence average converging to \(\beta_\phi(B)\).
Pairwise finite exact reachability inside \(B\) is a sufficient
condition for this conclusion at every state of \(B\).

Section~\ref{sec:setup} introduces the pathwise setup and the \(L\)-divergence decomposition. Section~\ref{sec:uniform-growth-invariant-obstructions} establishes the representations of the global rate. Section~\ref{sec:state-conditioned-values} represents the current-state lower asymptotic rate using empirical transition laws and derives bounds using basic sets. Section~\ref{sec:common-basic-set-rates} gives conditions for the restricted minimum average \(L\)-divergence to converge to \(\beta_\phi(B)\). Section~\ref{sec:discussion} discusses further questions. Appendix~\ref{app:component-pressure-tests} gives examples showing that both bounds can be strict. Appendix~\ref{app:three-asset-positive-growth} gives a three-asset example with a positive worst-case relative growth rate for the equal-weighted portfolio.

\section{Setup and \texorpdfstring{\(L\)}{L}-Divergence}\label{sec:setup}

Fix \(n\ge2\).  Let
\[
        \Delta_n
        :=
        \left\{
        p=(p_1,\dots,p_n)\in(0,\infty)^n:
        \sum_{i=1}^n p_i=1
        \right\}
\]
be the open unit simplex.  Its tangent space is
\[
        T\Delta_n
        :=
        \left\{
        v\in\mathbb R^n:
        \sum_{i=1}^n v_i=0
        \right\}.
\]

Throughout the paper, \(K\subset\Delta_n\) is a nonempty compact convex restricted market state space.  We assume that \(K\) is contained in an open convex set \(U\subset\Delta_n\), where openness is understood relative to the affine hyperplane containing the simplex.

For \(f\in C^1(U)\) and \(p\in U\), we use the tangential gradient \(\nabla f(p)\in T\Delta_n\), characterized by
\[
Df(p)[v]=\langle\nabla f(p),v\rangle,
\qquad v\in T\Delta_n.
\]
We write \(\partial_i f(p)\) for its \(i\)th component.

Let
\[
        K^{\mathbb N_0}
        =
        \{(\omega_0,\omega_1,\omega_2,\dots):\omega_k\in K\text{ for all }k\ge0\}
\]
be equipped with the product topology.  Then \(K^{\mathbb N_0}\) is compact and metrizable.  Fix a metric \(d_K\) compatible with the topology of \(K\).

Let
\[
        \sigma(\omega_0,\omega_1,\omega_2,\dots)
        =
        (\omega_1,\omega_2,\omega_3,\dots)
\]
be the left shift.

\begin{definition}[Restricted path space]
A restricted path space over \(K\) is a nonempty compact subset
\[
        \Omega\subset K^{\mathbb N_0}
\]
such that
\[
        \sigma(\Omega)\subset\Omega.
\]
We call this condition forward shift-invariance.
\end{definition}

The paths in \(\Omega\) are called admissible.  No deterministic one-step transition rule is assumed.

Let \(\mathcal P(\Omega)\) denote the set of Borel probability measures on \(\Omega\).

\begin{definition}[Shift-invariant measures]
A measure \(\nu\in\mathcal P(\Omega)\) is \(\sigma\)-invariant if
\[
        \int_\Omega h(\sigma\omega)\,d\nu(\omega)
        =
        \int_\Omega h(\omega)\,d\nu(\omega)
\]
for every \(h\in C(\Omega)\).  The set of all \(\sigma\)-invariant probability measures is denoted by \(\mathcal M_\sigma(\Omega)\).
\end{definition}

Since \(\Omega\) is compact and \(\sigma:\Omega\to\Omega\) is continuous, \(\mathcal M_\sigma(\Omega)\) is a nonempty compact convex subset of \(\mathcal P(\Omega)\) in the weak topology \cite[Theorem~6.5, Corollary~6.9.1, and Theorem~6.10]{Walters1982ErgodicTheory}.

\subsection{Functionally generated portfolios and \texorpdfstring{\(L\)}{L}-divergence}

We use functional generation on the restricted state space.  The underlying discrete-time wealth decomposition and \(L\)-divergence viewpoint originate in Pal and Wong's pathwise theory \cite{PalWong2016Geometry}.  The restricted-domain formulation used here is also consistent with the restricted market-state framework of Xu and Cui \cite{XuCui2026}.

Let \(\Phi:U\to(0,\infty)\) be a \(C^1\) concave function.  Set
\[
        \phi:=\log\Phi .
\]
For \(p\in K\), the functionally generated portfolio associated with \(\Phi\) is defined by
\begin{equation}\label{eq:portfolio-definition}
        \frac{\pi_i(p)}{p_i}
        =
        1+\partial_i\phi(p)-\sum_{j=1}^n p_j\partial_j\phi(p),
        \qquad i=1,\dots,n .
\end{equation}
Then \(\sum_i\pi_i(p)=1\).

For a transition from \(p\) to \(q\), with \(p,q\in K\), the one-step relative wealth factor is
\[
        \sum_{i=1}^n \pi_i(p)\frac{q_i}{p_i}
        =
        1+\langle\nabla\phi(p),q-p\rangle.
\]
The equality follows from the portfolio formula and
\(\sum_i(q_i-p_i)=0\).
Concavity and positivity of \(\Phi\) imply
\[
        \Phi(q)
        \le
        \Phi(p)+\langle\nabla\Phi(p),q-p\rangle
        =
        \Phi(p)\bigl(1+\langle\nabla\phi(p),q-p\rangle\bigr),
\]
so \(1+\langle\nabla\phi(p),q-p\rangle>0\) for all \(p,q\in K\).
The portfolio weights may be negative; the positivity just proved
ensures positive wealth along all admissible paths in the restricted
state space.

\begin{definition}[\(L\)-divergence]
For \(p,q\in K\), define
\begin{equation}\label{eq:L-divergence}
        T_\phi(q\mid p)
        :=
        \log\left(1+\langle\nabla\phi(p),q-p\rangle\right)
        -
        \bigl(\phi(q)-\phi(p)\bigr).
\end{equation}
\end{definition}

This is the restricted-domain version of the \(L\)-divergence introduced by Pal and Wong \cite{PalWong2016Geometry}.  By concavity of \(\Phi=e^\phi\), the quantity in \eqref{eq:L-divergence} is well-defined and nonnegative on \(K\times K\).  It measures the one-step excess log-return in the pathwise decomposition.

For \(\omega\in\Omega\) and \(t\ge1\), define the relative wealth of the portfolio with respect to the market over the first \(t\) periods along \(\omega\) by
\begin{equation}\label{eq:wealth-product}
        V_\pi(t,\omega)
        :=
        \prod_{k=0}^{t-1}
        \sum_{i=1}^n \pi_i(\omega_k)
        \frac{(\omega_{k+1})_i}{(\omega_k)_i}.
\end{equation}
Thus \(\log V_\pi(t,\omega)\) is the logarithm of relative wealth, and \(t^{-1}\log V_\pi(t,\omega)\) is the average logarithmic growth relative to the market over the first \(t\) periods.  Combining the definition of \(V_\pi\) with \eqref{eq:L-divergence} gives the pathwise wealth decomposition
\begin{equation}\label{eq:wealth-decomposition}
        \log V_\pi(t,\omega)
        =
        \phi(\omega_t)-\phi(\omega_0)
        +
        \sum_{k=0}^{t-1}
        T_\phi(\omega_{k+1}\mid \omega_k).
\end{equation}

\subsection{The one-step \texorpdfstring{\(L\)}{L}-divergence function}

The wealth decomposition singles out the \(L\)-divergence associated with each transition \(p\to q\).  On the restricted path space, define
\begin{equation}\label{eq:L-divergence-function}
        f_\phi(\omega)
        :=
        T_\phi(\omega_1\mid \omega_0),
        \qquad \omega\in\Omega .
\end{equation}
Since \(\Phi\in C^1(U)\), \(K\) is compact, and the coordinate maps \(\omega\mapsto\omega_0\), \(\omega\mapsto\omega_1\) are continuous, we have \(f_\phi\in C(\Omega)\).  Moreover,
\begin{equation}\label{eq:birkhoff-transition-average}
        \frac1t\sum_{k=0}^{t-1}f_\phi(\sigma^k\omega)
        =
        \frac1t\sum_{k=0}^{t-1}
        T_\phi(\omega_{k+1}\mid\omega_k).
\end{equation}
Thus the Birkhoff average of \(f_\phi\) is exactly the average one-step \(L\)-divergence accumulated along the admissible market path \(\omega\).

The pathwise wealth decomposition can be written as
\begin{equation}\label{eq:wealth-decomposition-L-divergence}
        \log V_\pi(t,\omega)
        =
        \phi(\omega_t)-\phi(\omega_0)
        +
        \sum_{k=0}^{t-1} f_\phi(\sigma^k\omega).
\end{equation}
\section{Global Worst-Case Relative Growth}\label{sec:uniform-growth-invariant-obstructions}

We begin with the global problem, in which the current market state is not fixed.  We first represent the global rate by an invariant path-measure average.  When admissibility is generated by a compact transition relation, we pass to balanced transition measures on that relation and then localize a global minimizing measure to one relation basic set.

\subsection{Invariant-measure representation}

We first express the worst-case relative growth rate as a minimum over invariant measures. The proof uses the minimizing form of the standard ergodic-optimization identity \cite[Proposition~2.1]{Jenkinson2006}.

\begin{definition}[Worst-case relative growth rate]\label{def:worst-case-growth-rate}
For a functionally generated portfolio \(\pi\) and a restricted path space \(\Omega\), define
\begin{equation}\label{eq:worst-case-growth-rate}
        \gamma(\Phi,\Omega)
        :=
        \lim_{t\to\infty}
        \min_{\omega\in\Omega}
        \frac1t\log V_\pi(t,\omega).
\end{equation}
\end{definition}

The existence of the finite limit in \eqref{eq:worst-case-growth-rate} is not assumed a priori.  Lemma~\ref{lem:worst-case-birkhoff-averages} establishes the limit of the minimum finite-horizon average of a continuous function. Theorem~\ref{thm:invariant-L-divergence-representation} then identifies \(\gamma(\Phi,\Omega)\) with the minimum invariant average of \(f_\phi\).

\begin{lemma}[Worst-case Birkhoff averages]\label{lem:worst-case-birkhoff-averages}
Let \(f\in C(\Omega)\).  Then
\begin{equation}\label{eq:worst-case-birkhoff-variational}
        \lim_{t\to\infty}
        \frac1t
        \min_{\omega\in\Omega}
        \sum_{k=0}^{t-1}f(\sigma^k\omega)
        =
        \min_{\nu\in\mathcal M_\sigma(\Omega)}
        \int_\Omega f\,d\nu .
\end{equation}
\end{lemma}

\begin{proof}
For each fixed \(t\), the function \(\omega\mapsto\sum_{k=0}^{t-1}f(\sigma^k\omega)\) is continuous on compact \(\Omega\), so its minimum is attained. By forward shift-invariance, the sequence
\[
\left(\min_{\omega\in\Omega}\sum_{k=0}^{t-1}f(\sigma^k\omega)\right)_{t\ge1}
\]
is superadditive.  Since \(f\) is bounded on compact \(\Omega\), Fekete's lemma gives a finite limit after division by \(t\).

For every \(\nu\in\mathcal M_\sigma(\Omega)\), shift invariance gives
\[
\frac1t\min_{\omega\in\Omega}\sum_{k=0}^{t-1}f(\sigma^k\omega)
\le
\int_\Omega \frac1t\sum_{k=0}^{t-1}f(\sigma^k\omega)\,d\nu(\omega)
=
\int_\Omega f\,d\nu,
\]
so the limit is no larger than the minimum invariant average.

For the reverse inequality, choose a minimizing path \(\omega^{(j)}\in\Omega\) for each \(t_j\to\infty\), and set
\[
\mu_j:=\frac1{t_j}\sum_{k=0}^{t_j-1}\delta_{\sigma^k\omega^{(j)}}.
\]
After passing to a subsequence, compactness of \(\mathcal P(\Omega)\) gives \(\mu_j\Longrightarrow\mu\).  For every \(h\in C(\Omega)\),
\[
\int_\Omega(h\circ\sigma-h)\,d\mu_j
=
\frac{h(\sigma^{t_j}\omega^{(j)})-h(\omega^{(j)})}{t_j}
\longrightarrow0,
\]
hence \(\mu\in\mathcal M_\sigma(\Omega)\).  Moreover,
\[
\int_\Omega f\,d\mu
=
\lim_{j\to\infty}\frac1{t_j}\min_{\omega\in\Omega}\sum_{k=0}^{t_j-1}f(\sigma^k\omega)
=
\lim_{t\to\infty}\frac1t\min_{\omega\in\Omega}\sum_{k=0}^{t-1}f(\sigma^k\omega).
\]
Thus the minimum invariant average is no larger than the finite-horizon limit, proving \eqref{eq:worst-case-birkhoff-variational}.
\end{proof}

\begin{theorem}[Invariant \(L\)-divergence representation]\label{thm:invariant-L-divergence-representation}
Let \(K\subset\Delta_n\) be compact and convex, let \(\Omega\) be a restricted path space over \(K\), and let \(\Phi\) be a positive \(C^1\) concave generating function on a relative-open convex neighborhood of \(K\).  Set \(\phi=\log\Phi\), and let \(\pi\) be the associated functionally generated portfolio.  Then
\begin{equation}\label{eq:growth-rate-variational}
        \gamma(\Phi,\Omega)
        =
        \min_{\nu\in\mathcal M_\sigma(\Omega)}
        \int_\Omega f_\phi\,d\nu .
\end{equation}
\end{theorem}

\begin{proof}
By the wealth decomposition \eqref{eq:wealth-decomposition-L-divergence},
\[
\left|
\min_{\omega\in\Omega}\frac1t\log V_\pi(t,\omega)
-
\frac1t\min_{\omega\in\Omega}\sum_{k=0}^{t-1}f_\phi(\sigma^k\omega)
\right|
\le
\frac{\max_{p\in K}\phi(p)-\min_{p\in K}\phi(p)}{t}.
\]
The endpoint bound tends to zero, and Lemma~\ref{lem:worst-case-birkhoff-averages} identifies the limit of the accumulated \(L\)-divergence minima with the minimum invariant average.  This proves \eqref{eq:growth-rate-variational}.
\end{proof}

\subsection{Relation-generated paths and balanced transition measures}\label{sec:transition-level-obstructions-and-certificates}

We now specialize to a compact transition relation \(\Gamma\subset K\times K\), which allows the global problem to be represented directly on the transition space \(\Gamma\).

Define
\begin{equation}\label{eq:Omega-Gamma}
        \Omega_\Gamma
        :=
        \{\omega\in K^{\mathbb N_0}:(\omega_k,\omega_{k+1})\in\Gamma
        \text{ for all }k\ge0\}.
\end{equation}
Throughout the remainder of this section, \(\Gamma\subset K\times K\) is compact and \(\Omega_\Gamma\neq\varnothing\).  Thus \(\Omega_\Gamma\) is a compact forward shift-invariant restricted path space.
For the closed-relation arguments below, we also use the bilateral sample-path space
\begin{equation}\label{eq:bilateral-path-space}
        X_\Gamma
        :=
        \left\{
        \widehat\omega\in K^{\mathbb Z}:
        (\widehat\omega_k,\widehat\omega_{k+1})\in\Gamma
        \text{ for every }k\in\mathbb Z
        \right\}.
\end{equation}
Compactness and \(\Omega_\Gamma\ne\varnothing\) imply that \(X_\Gamma\ne\varnothing\).
Let \(s_\Gamma:X_\Gamma\to X_\Gamma\) be the bilateral shift,
\[
        s_\Gamma((\widehat\omega_k)_{k\in\mathbb Z})
        :=
        (\widehat\omega_{k+1})_{k\in\mathbb Z}.
\]
We use one projection convention throughout:
\[
\operatorname{pr}_k^+:\Omega_\Gamma\to K,\quad \operatorname{pr}_k^+(\omega)=\omega_k\quad(k\ge0),
\qquad
\operatorname{pr}_k^{\mathbb Z}:X_\Gamma\to K,\quad \operatorname{pr}_k^{\mathbb Z}(\widehat\omega)=\widehat\omega_k\quad(k\in\mathbb Z),
\]
and write
\[
\operatorname{pr}_{0,1}^+:=(\operatorname{pr}_0^+,\operatorname{pr}_1^+):\Omega_\Gamma\to\Gamma,
\qquad
\operatorname{pr}_{0,1}^{\mathbb Z}:=(\operatorname{pr}_0^{\mathbb Z},\operatorname{pr}_1^{\mathbb Z}):X_\Gamma\to\Gamma.
\]
We write \(\mathcal M_{s_\Gamma}(X_\Gamma)\) for the set of \(s_\Gamma\)-invariant Borel probability measures on \(X_\Gamma\).  On the edge space \(\Gamma\), \(\mathrm{pr}_1,\mathrm{pr}_2:\Gamma\to K\) always denote the two coordinate projections.

\begin{definition}[Balanced transition measures]\label{def:balanced-transition-measures}
Let \(\mathcal I(\Gamma)\) be the set of Borel probability measures \(\lambda\) on \(\Gamma\) such that
\begin{equation}\label{eq:balanced-transition-condition}
        (\mathrm{pr}_1)_\#\lambda
        =
        (\mathrm{pr}_2)_\#\lambda .
\end{equation}
The elements of \(\mathcal I(\Gamma)\) are called \emph{balanced transition measures}.
\end{definition}

The balance condition \eqref{eq:balanced-transition-condition} is equivalent to
\begin{equation}\label{eq:balanced-transition-test-functions}
        \int_\Gamma h(p)\,d\lambda(p,q)
        =
        \int_\Gamma h(q)\,d\lambda(p,q),
        \qquad h\in C(K).
\end{equation}
Thus \(\lambda\) is a probability law on admissible one-step transitions whose current-state and next-state marginals coincide.  Its common marginal satisfies the invariant-measure condition for the closed relation in \cite[Theorem~3.2, condition~(3)]{MillerAkin1999}.  The following lemma lifts each balanced transition measure to a bilateral shift-invariant path measure with the prescribed \((0,1)\)-marginal.
A related graph-system formulation appears in \cite[Lemma~4.10]{JenkinsonLiLiaoZhang2025}.

\begin{lemma}[Bilateral representation of balanced measures]\label{lem:bilateral-path-transition-correspondence}
The map \(\widehat\nu\mapsto(\operatorname{pr}_{0,1}^{\mathbb Z})_\#\widehat\nu\) sends \(\mathcal M_{s_\Gamma}(X_\Gamma)\) onto \(\mathcal I(\Gamma)\).  Equivalently,
\begin{equation}\label{eq:bilateral-path-transition-surjection}
        (\operatorname{pr}_{0,1}^{\mathbb Z})_\#\mathcal M_{s_\Gamma}(X_\Gamma)
        =
        \mathcal I(\Gamma).
\end{equation}
\end{lemma}

\begin{proof}
If \(\widehat\nu\in\mathcal M_{s_\Gamma}(X_\Gamma)\) and
\(\lambda=(\operatorname{pr}_{0,1}^{\mathbb Z})_\#\widehat\nu\), then \(\lambda\) is supported on \(\Gamma\), and shift invariance gives, for every \(h\in C(K)\),
\[
\int_\Gamma h(p)\,d\lambda(p,q)
=
\int_{X_\Gamma}h(\widehat\omega_0)\,d\widehat\nu
=
\int_{X_\Gamma}h(\widehat\omega_1)\,d\widehat\nu
=
\int_\Gamma h(q)\,d\lambda(p,q).
\]
Thus \(\lambda\in\mathcal I(\Gamma)\).

Conversely, let \(\lambda\in\mathcal I(\Gamma)\) and let \(m\) be its common marginal.  Disintegrate the given coupling as
\[
\lambda(dp,dq)=m(dp)\,\kappa_p(dq).
\]
Equality of the two marginals means that
\[
\int_K\kappa_p(A)\,m(dp)=\lambda(K\times A)=m(A)
\qquad\text{for every Borel }A\subseteq K.
\]
Hence \(m\) is stationary for the transition kernel \(\kappa\).  As in \cite[Theorems~3.1--3.2]{MillerAkin1999}, consider the finite-dimensional laws
\[
m(dp_a)\,\kappa_{p_a}(dp_{a+1})\cdots\kappa_{p_{b-1}}(dp_b),
\qquad a\le b.
\]
These finite-dimensional laws are consistent because \(\kappa\) is a probability kernel and \(m\) is stationary. Kolmogorov extension therefore yields a stationary law \(\widehat\nu\) on \(K^{\mathbb Z}\) with \((0,1)\)-marginal \(\lambda\). By stationarity, the \((k,k+1)\)-marginal is \(\lambda\) for every \(k\in\mathbb Z\). Since \(\lambda(\Gamma)=1\), the countable intersection of the events \(\{(\widehat\omega_k,\widehat\omega_{k+1})\in\Gamma\}\), \(k\in\mathbb Z\), has full \(\widehat\nu\)-measure; hence \(\widehat\nu(X_\Gamma)=1\). Thus \(\widehat\nu\in\mathcal M_{s_\Gamma}(X_\Gamma)\). Therefore
\(
(\operatorname{pr}_{0,1}^{\mathbb Z})_\#\widehat\nu=\lambda
\), proving \eqref{eq:bilateral-path-transition-surjection}.
\end{proof}

\begin{theorem}[Transition-measure representation]\label{thm:transition-measure-form}
Let \(\Gamma\subset K\times K\) be compact, and suppose that \(\Omega_\Gamma\neq\varnothing\).  Then
\begin{equation}\label{eq:forward-path-transition-correspondence}
        (\operatorname{pr}_{0,1}^{+})_\#\mathcal M_\sigma(\Omega_\Gamma)
        =
        \mathcal I(\Gamma),
\end{equation}
and
\begin{equation}\label{eq:transition-L-divergence-criterion}
        \gamma(\Phi,\Omega_\Gamma)
        =
        \min_{\lambda\in\mathcal I(\Gamma)}
        \int_\Gamma T_\phi(q\mid p)\,d\lambda(p,q).
\end{equation}
\end{theorem}

\begin{proof}
If \(\nu\in\mathcal M_\sigma(\Omega_\Gamma)\), then
\(\lambda=(\operatorname{pr}_{0,1}^{+})_\#\nu\) is supported on \(\Gamma\), and shift invariance gives equal first and second marginals.  Hence the left-hand side of \eqref{eq:forward-path-transition-correspondence} is contained in \(\mathcal I(\Gamma)\).

Conversely, for \(\lambda\in\mathcal I(\Gamma)\), Lemma~\ref{lem:bilateral-path-transition-correspondence} gives \(\widehat\nu\in\mathcal M_{s_\Gamma}(X_\Gamma)\) with prescribed \((0,1)\)-law \(\lambda\).  Projecting \(\widehat\nu\) to its nonnegative coordinates gives \(\nu\in\mathcal M_\sigma(\Omega_\Gamma)\) with
\((\operatorname{pr}_{0,1}^{+})_\#\nu=\lambda\).  This proves \eqref{eq:forward-path-transition-correspondence}.  It also shows that \(\mathcal I(\Gamma)\) is the continuous image of the nonempty compact set \(\mathcal M_\sigma(\Omega_\Gamma)\), and is therefore nonempty and compact.

Finally,
\[
\int_{\Omega_\Gamma}f_\phi\,d\nu
=
\int_\Gamma T_\phi(q\mid p)\,
 d((\operatorname{pr}_{0,1}^{+})_\#\nu)(p,q).
\]
Combining \eqref{eq:forward-path-transition-correspondence} with Theorem~\ref{thm:invariant-L-divergence-representation} gives \eqref{eq:transition-L-divergence-criterion}.
\end{proof}

\begin{remark}[Criterion for a zero global rate]\label{rem:zero-L-balanced-obstruction}
Since \(T_\phi\ge0\) and the minimum in \eqref{eq:transition-L-divergence-criterion} is attained, that formula gives
\[
        \gamma(\Phi,\Omega_\Gamma)=0
\]
if and only if there exists \(\lambda\in\mathcal I(\Gamma)\) such that
\[
        T_\phi(y\mid x)=0
        \qquad\text{for \(\lambda\)-almost every }(x,y)\in\Gamma.
\]
For example, an admissible self-loop \((p,p)\in\Gamma\) gives the balanced measure \(\delta_{(p,p)}\), whose \(L\)-divergence average is zero.
\end{remark}

\subsection{Recurrent localization and basic sets}\label{sec:chain-recurrent-localization}

We next locate the supports of balanced transition measures using chain recurrence. The closed-relation framework of Miller and Akin \cite[Sections~4--5]{MillerAkin1999} then allows us to select a global minimizer supported on one basic set.

\begin{definition}[Chain recurrence and basic sets]
For \(p\in K\), write \(\Gamma(p):=\{q\in K:(p,q)\in\Gamma\}\).  Given \(\varepsilon>0\), an \(\varepsilon\)-chain from \(p\) to \(q\) is a finite sequence \(p=\xi_0,\xi_1,\ldots,\xi_N=q\), with \(N\ge1\), such that for every \(0\le i<N\) there is \(\zeta_i\in\Gamma(\xi_i)\) satisfying
\[
d_K(\xi_{i+1},\zeta_i)<\varepsilon.
\]
The chain relation \(\mathcal C_\Gamma\) consists of the pairs \((p,q)\) for which such a chain exists for every \(\varepsilon>0\).  Set
\[
\mathcal R(\Gamma):=\{p\in K:(p,p)\in\mathcal C_\Gamma\},
\qquad
\mathcal C_\Gamma(p):=\{q\in K:(p,q)\in\mathcal C_\Gamma\}.
\]
These are the chain-recurrent set and the chain-forward set of \(p\), respectively.  On \(\mathcal R(\Gamma)\), define
\[
p\sim_\Gamma q
\quad\Longleftrightarrow\quad
(p,q)\in\mathcal C_\Gamma\ \text{and}\ (q,p)\in\mathcal C_\Gamma.
\]
The equivalence classes are the basic sets of \(\Gamma\); their family is denoted by \(\mathcal B(\Gamma)\).
\end{definition}

For \(Y\subseteq K\), write \(\Gamma_Y:=\Gamma\cap(Y\times Y)\).  We call \(\Gamma_Y\) surjective on \(Y\) if every point of \(Y\) has both a \(\Gamma_Y\)-predecessor and a \(\Gamma_Y\)-successor.

\begin{lemma}[Basic sets and invariant path measures]\label{lem:basic-set-path-measures}
The chain relation \(\mathcal C_\Gamma\) is closed.  Each \(B\in\mathcal B(\Gamma)\) is compact, and \(\Gamma_B\) is surjective on \(B\).  For every \(\widehat\nu\in\mathcal M_{s_\Gamma}(X_\Gamma)\), there are a Borel probability measure \(\Pi\) on \(\mathcal M_{s_\Gamma}(X_\Gamma)\) and a Borel set \(E\subseteq\mathcal M_{s_\Gamma}(X_\Gamma)\)
with \(\Pi(E)=1\) such that
\begin{equation}\label{eq:basic-set-average-decomposition}
\widehat\nu(A)=\int_E\eta(A)\,\Pi(d\eta)
\qquad\text{for every Borel }A\subseteq X_\Gamma,
\end{equation}
and each \(\eta\in E\) is supported on \(X_\Gamma\cap B_\eta^{\mathbb Z}\) for some \(B_\eta\in\mathcal B(\Gamma)\).
\end{lemma}

\begin{proof}
The closedness and basic-set properties follow from
\cite[Section~4 and Lemma~4.1]{MillerAkin1999}.
To compare the basic sets of \(\Gamma\) with those of the bilateral shift \(s_\Gamma\), first restrict \(\Gamma\) to its dynamic domain
\[
D_\Gamma:=\operatorname{pr}_0^{\mathbb Z}(X_\Gamma).
\]
This set is compact. Every point of \(D_\Gamma\) lies on a bilateral path, whose adjacent states give a predecessor and a successor in \(D_\Gamma\); hence \(\Gamma_{D_\Gamma}\) is surjective on \(D_\Gamma\). By \cite[Lemmas~1.2 and~4.1, and equation~(4.16)]{MillerAkin1999},
\[
\mathcal R(\Gamma)\subseteq D_\Gamma,
\qquad
\mathcal C_{\Gamma_{D_\Gamma}}
=\mathcal C_\Gamma\cap(D_\Gamma\times D_\Gamma).
\]
Thus \(\Gamma_{D_\Gamma}\) has the same chain-recurrent set and basic sets as \(\Gamma\), and its bilateral path space is also \(X_\Gamma\). Applying \cite[Theorem~4.3(c)]{MillerAkin1999} to \(\Gamma_{D_\Gamma}\) identifies the basic sets of the bilateral shift with \(X_\Gamma\cap B^{\mathbb Z}\), where \(B\in\mathcal B(\Gamma)\).

Now take an ergodic decomposition of \(\widehat\nu\), with decomposition
measure \(\Pi\) and a Borel set \(E\) of ergodic
\(s_\Gamma\)-invariant measures satisfying \(\Pi(E)=1\).
By \cite[Corollary~5.4(a)]{MillerAkin1999}, each \(\eta\in E\)
is supported on one basic set of the bilateral shift.
The preceding correspondence gives the stated support in
\(X_\Gamma\cap B_\eta^{\mathbb Z}\), and the ergodic decomposition
gives \eqref{eq:basic-set-average-decomposition}.
\end{proof}

\begin{proposition}[Recurrent localization of balanced transition measures]\label{prop:balanced-measures-chain-recurrent}
For every \(\lambda\in\mathcal I(\Gamma)\),
\begin{equation}\label{eq:balanced-measure-chain-recurrent-support}
        \lambda\bigl(\mathcal R(\Gamma)\times\mathcal R(\Gamma)\bigr)=1.
\end{equation}
\end{proposition}

\begin{proof}
Lift \(\lambda\) by Lemma~\ref{lem:bilateral-path-transition-correspondence} to \(\widehat\nu\in\mathcal M_{s_\Gamma}(X_\Gamma)\), and let \(m\) be the common marginal of \(\lambda\).  Apply \eqref{eq:basic-set-average-decomposition}.  For \(\Pi\)-almost every \(\eta\), the corresponding basic set \(B_\eta\) is contained in \(\mathcal R(\Gamma)\), so
\[
\eta\bigl((\operatorname{pr}_0^{\mathbb Z})^{-1}(\mathcal R(\Gamma))\bigr)=1.
\]
Since \(m=(\operatorname{pr}_0^{\mathbb Z})_\#\widehat\nu\), the decomposition gives
\[
m(\mathcal R(\Gamma))
=\int_E\eta\bigl((\operatorname{pr}_0^{\mathbb Z})^{-1}(\mathcal R(\Gamma))\bigr)\,\Pi(d\eta)
=1.
\]
Since both marginals of \(\lambda\) equal \(m\), the measure \(\lambda\) assigns full mass to \(\mathcal R(\Gamma)\times\mathcal R(\Gamma)\), which is \eqref{eq:balanced-measure-chain-recurrent-support}.
\end{proof}

\begin{corollary}[A global minimizer supported on one basic set]\label{cor:basic-set-global-minimizer}
There exist a basic set \(B\subset\mathcal R(\Gamma)\) and a measure
\(\lambda_B\in\mathcal I(\Gamma)\) such that \(\lambda_B\) attains the minimum in
\eqref{eq:transition-L-divergence-criterion} and
\begin{equation}\label{eq:basic-set-minimizer-support}
        \lambda_B(B\times B)=1.
\end{equation}
\end{corollary}

\begin{proof}
Choose \(\lambda^\ast\in\mathcal I(\Gamma)\) attaining the minimum in
\eqref{eq:transition-L-divergence-criterion}, and lift it by
Lemma~\ref{lem:bilateral-path-transition-correspondence} to
\(\widehat\nu^\ast\in\mathcal M_{s_\Gamma}(X_\Gamma)\).  Apply \eqref{eq:basic-set-average-decomposition} to \(\widehat\nu^\ast\).  For \(\eta\in E\), let \(B_\eta\in\mathcal B(\Gamma)\) be the corresponding basic set and set
\[
\lambda_\eta:=(\operatorname{pr}_{0,1}^{\mathbb Z})_\#\eta.
\]
Then \(\lambda_\eta\in\mathcal I(\Gamma_{B_\eta})\).  Applying \eqref{eq:basic-set-average-decomposition} to \((\operatorname{pr}_{0,1}^{\mathbb Z})^{-1}(A)\), for Borel \(A\subseteq\Gamma\), gives
\[
\lambda^\ast=\int_E\lambda_\eta\,\Pi(d\eta).
\]
Hence
\[
\gamma(\Phi,\Omega_\Gamma)
=
\int_E\left(\int_{\Gamma_{B_\eta}}T_\phi(y\mid x)\,d\lambda_\eta(x,y)\right)\Pi(d\eta).
\]
Each inner integral is at least the global minimum \(\gamma(\Phi,\Omega_\Gamma)\).  Therefore equality holds for \(\Pi\)-almost every \(\eta\).  Choose one such \(\eta\) and write \(B=B_\eta\), \(\lambda_B=\lambda_\eta\).  Then \(\lambda_B\) attains the global minimum and \(\lambda_B(B\times B)=1\), proving \eqref{eq:basic-set-minimizer-support}.
\end{proof}

For every \(B\in\mathcal B(\Gamma)\),
Lemma~\ref{lem:basic-set-path-measures} shows that \(\Gamma_B\)
is surjective on \(B\). In particular, every point of \(B\)
has a \(\Gamma_B\)-successor, so an infinite admissible path
can be constructed within \(B\) and \(\Omega_{\Gamma_B}\neq\varnothing\).  Theorem~\ref{thm:transition-measure-form} therefore shows that \(\mathcal I(\Gamma_B)\) is nonempty and compact.  Define
\begin{equation}\label{eq:component-stationary-floor}
\beta_\phi(B)
:=
\min_{\lambda\in\mathcal I(\Gamma_B)}
\int_{\Gamma_B}T_\phi(y\mid x)\,d\lambda(x,y).
\end{equation}
The minimum is attained by continuity of \(T_\phi\).  Since \(\mathcal I(\Gamma_B)\subseteq\mathcal I(\Gamma)\), each \(\beta_\phi(B)\) is at least \(\gamma(\Phi,\Omega_\Gamma)\).  Corollary~\ref{cor:basic-set-global-minimizer} gives equality for at least one \(B\), and hence
\begin{equation}\label{eq:global-basic-set-minimum}
\gamma(\Phi,\Omega_\Gamma)
=\min_{B\in\mathcal B(\Gamma)}\beta_\phi(B).
\end{equation}

\section{Long-Run Growth from the Current Market State}\label{sec:state-conditioned-values}

We now fix the current market state as the time-zero state \(p\).  The relevant long-run transition laws are the empirical limits generated by arbitrarily long admissible paths starting from \(p\).  We represent the resulting lower asymptotic rate through these laws and then obtain lower and upper bounds using relation basic sets.

\subsection{Current-state values and empirical transition laws}

Recall that \(\Gamma\subset K\times K\) is compact and \(\Omega_\Gamma\ne\varnothing\). These assumptions remain in force throughout this section. The set of admissible initial states is
\begin{equation}\label{eq:viable-state-set}
        K_\Gamma^\infty
        :=
        \operatorname{pr}_0^+(\Omega_\Gamma).
\end{equation}
Thus \(K_\Gamma^\infty\) is the compact set of states from which an infinite \(\Gamma\)-admissible path starts.  Throughout this section, initial states are understood to lie in \(K_\Gamma^\infty\).  For such \(p\), the set of admissible paths starting from \(p\) is nonempty and is denoted by
\begin{equation}\label{eq:initial-state-path-fiber}
        \Omega_\Gamma(p)
        :=
        \{\omega\in\Omega_\Gamma:\omega_0=p\}.
\end{equation}

For \(t\ge1\), define the minimum accumulated \(L\)-divergence over the first \(t\) periods among admissible paths starting from \(p\) by
\begin{equation}\label{eq:state-conditioned-finite-horizon-value}
        \mathcal V_t^\phi(p)
        :=
        \min_{\omega\in\Omega_\Gamma(p)}
        \sum_{k=0}^{t-1}T_\phi(\omega_{k+1}\mid\omega_k),
\end{equation}
and set \(\mathcal V_0^\phi\equiv0\).  The minimum is attained because \(\Omega_\Gamma(p)\) is compact and the objective is continuous and depends on finitely many coordinates.  The minimum of the logarithm of relative wealth over \(\Omega_\Gamma(p)\) is also attained, and the wealth decomposition gives
\begin{equation}\label{eq:state-conditioned-endpoint-bound}
        \left|
        \min_{\omega\in\Omega_\Gamma(p)}\log V_\pi(t,\omega)
        -\mathcal V_t^\phi(p)
        \right|
        \le
        \max_{x\in K}\phi(x)-\min_{x\in K}\phi(x).
\end{equation}
Dividing this bound by \(t\) and taking lower limits gives
\[
        \liminf_{t\to\infty}
        \frac1t\min_{\omega\in\Omega_\Gamma(p)}\log V_\pi(t,\omega)
        =
        \liminf_{t\to\infty}\frac{\mathcal V_t^\phi(p)}{t}.
\]
We therefore work with \(\mathcal V_t^\phi(p)/t\) below.

\begin{definition}[Empirical transition laws]\label{def:current-state-empirical-transition-laws}
For \(p\in K_\Gamma^\infty\), \(\omega\in\Omega_\Gamma(p)\), and \(t\ge1\), define
\begin{equation}\label{eq:empirical-transition-measure-from-state}
L_t^\omega:=\frac1t\sum_{k=0}^{t-1}\delta_{(\omega_k,\omega_{k+1})}.
\end{equation}
The empirical transition laws generated from \(p\) form the set
\begin{equation}\label{eq:asymptotically-accessible-transition-measures}
\begin{aligned}
\mathfrak A_\Gamma(p):=\bigl\{\lambda\in\mathcal P(\Gamma):
&\text{there exist }t_j\to\infty\text{ and }\omega^{(j)}\in\Omega_\Gamma(p)\\
&\text{such that }L_{t_j}^{\omega^{(j)}}\Longrightarrow\lambda\bigr\}.
\end{aligned}
\end{equation}
\end{definition}

If
\[
        \mathcal E_t^\Gamma(p)
        :=
        \{L_t^\omega:\omega\in\Omega_\Gamma(p)\},
\]
then equivalently
\begin{equation}\label{eq:attainable-tail-closure}
        \mathfrak A_\Gamma(p)
        =
        \bigcap_{N\ge1}
        \overline{\bigcup_{t\ge N}\mathcal E_t^\Gamma(p)}
\end{equation}
in the weak topology of \(\mathcal P(\Gamma)\).

\begin{lemma}[Empirical transition limits are balanced]\label{lem:accessible-transition-measures-stationary}
For every \(p\in K_\Gamma^\infty\), the set \(\mathfrak A_\Gamma(p)\) is nonempty and compact, and
\begin{equation}\label{eq:accessible-measures-subset-stationary}
        \mathfrak A_\Gamma(p)\subset\mathcal I(\Gamma).
\end{equation}
\end{lemma}

\begin{proof}
For \(N\ge1\), set
\[
F_N:=\overline{\bigcup_{t\ge N}\mathcal E_t^\Gamma(p)}.
\]
Since \(p\in K_\Gamma^\infty\), the set \(\Omega_\Gamma(p)\) is nonempty; hence \(\mathcal E_t^\Gamma(p)\neq\varnothing\) for every \(t\), and therefore each \(F_N\) is nonempty.  Because \(\Gamma\) is compact, \(\mathcal P(\Gamma)\) is compact in the weak topology, so every \(F_N\) is compact.  Moreover,
\[
F_1\supseteq F_2\supseteq\cdots.
\]
Thus \((F_N)\) is a nested sequence of nonempty compact subsets of the compact space \(\mathcal P(\Gamma)\), so
\(
\bigcap_{N\ge1}F_N\neq\varnothing
\).
The intersection is closed in \(\mathcal P(\Gamma)\), hence compact.

Let \(\lambda\in\mathfrak A_\Gamma(p)\), and choose \(t_j\to\infty\) and \(\omega^{(j)}\in\Omega_\Gamma(p)\) such that \(L_{t_j}^{\omega^{(j)}}\Longrightarrow\lambda\).  For every \(h\in C(K)\),
\[
        \int_\Gamma\bigl(h(x)-h(y)\bigr)
        \,dL_{t_j}^{\omega^{(j)}}(x,y)
        =
        \frac{h(p)-h(\omega_{t_j}^{(j)})}{t_j}.
\]
The terminal difference divided by \(t_j\) tends to zero.  Passing to the weak limit gives
\[
        \int_\Gamma h(x)\,d\lambda(x,y)
        =
        \int_\Gamma h(y)\,d\lambda(x,y),
        \qquad h\in C(K),
\]
which is exactly the equal-marginal condition.
\end{proof}

Define the current-state lower asymptotic rate
\begin{equation}\label{eq:state-conditioned-liminf-value}
        \underline\alpha_\phi(p)
        :=
        \liminf_{t\to\infty}
        \frac{\mathcal V_t^\phi(p)}{t},
        \qquad p\in K_\Gamma^\infty.
\end{equation}

\begin{remark}[Why the current-state value retains a lower limit]\label{rem:fixed-initial-liminf}
For \(m,n\ge1\),
\[
\mathcal V_{m+n}^\phi(p)
=
\min_{\omega\in\Omega_\Gamma(p)}
\left\{
\sum_{k=0}^{m-1}T_\phi(\omega_{k+1}\mid\omega_k)
+
\mathcal V_n^\phi(\omega_m)
\right\}.
\]
The remaining minimization starts from \(\omega_m\), rather than from \(p\).
Thus the superadditivity argument for the global value does not apply
directly to \(\mathcal V_t^\phi(p)\), and the global Fekete argument
does not establish convergence of \(\mathcal V_t^\phi(p)/t\).

For every \(p\in K_\Gamma^\infty\), one always has
\[
\begin{aligned}
\underline\alpha_\phi(p)
&=\liminf_{t\to\infty}\min_{\omega\in\Omega_\Gamma(p)}
\frac1t\sum_{k=0}^{t-1}T_\phi(\omega_{k+1}\mid\omega_k)\\
&\le\inf_{\omega\in\Omega_\Gamma(p)}\liminf_{t\to\infty}
\frac1t\sum_{k=0}^{t-1}T_\phi(\omega_{k+1}\mid\omega_k).
\end{aligned}
\]
Indeed, the finite-horizon minimum is bounded above by the corresponding average along any fixed admissible path. Taking lower limits and then the infimum over paths gives the inequality. Interchanging optimization and the lower limit requires additional justification, as also noted in the control setting of \cite[equations~(0.2)--(0.3)]{BorkarGaitsgoryShvartsman2022Erratum}.
\end{remark}

\begin{theorem}[Current-state empirical-limit representation]\label{thm:state-conditioned-occupation-measure-representation}
For every \(p\in K_\Gamma^\infty\),
\begin{equation}\label{eq:state-conditioned-occupation-measure-formula}
        \underline\alpha_\phi(p)
        =
        \min_{\lambda\in\mathfrak A_\Gamma(p)}
        \int_\Gamma T_\phi(y\mid x)\,d\lambda(x,y).
\end{equation}
\end{theorem}

\begin{proof}
Let \(\lambda\in\mathfrak A_\Gamma(p)\), and choose \(t_j\to\infty\) and \(\omega^{(j)}\in\Omega_\Gamma(p)\) such that \(L_{t_j}^{\omega^{(j)}}\Longrightarrow\lambda\).  Since \(\mathcal V_{t_j}^\phi(p)\) is the minimum over paths in \(\Omega_\Gamma(p)\),
\[
        \int_\Gamma T_\phi(y\mid x)\,dL_{t_j}^{\omega^{(j)}}(x,y)
        \ge
        \frac{\mathcal V_{t_j}^\phi(p)}{t_j}.
\]
Taking limits and using continuity of \(T_\phi\) gives
\[
        \int_\Gamma T_\phi(y\mid x)\,d\lambda(x,y)
        \ge
        \underline\alpha_\phi(p).
\]

For the reverse inequality, choose \(t_j\to\infty\) such that
\[
        \frac{\mathcal V_{t_j}^\phi(p)}{t_j}
        \longrightarrow
        \underline\alpha_\phi(p),
\]
and choose a minimizing path \(\omega^{(j)}\in\Omega_\Gamma(p)\) for each \(t_j\).  After passing to a subsequence,
\(L_{t_j}^{\omega^{(j)}}\Longrightarrow\lambda_\ast\) for some \(\lambda_\ast\in\mathfrak A_\Gamma(p)\).  Therefore,
\[
        \int_\Gamma T_\phi(y\mid x)\,d\lambda_\ast(x,y)
        =
        \lim_{j\to\infty}
        \int_\Gamma T_\phi(y\mid x)\,dL_{t_j}^{\omega^{(j)}}(x,y)
        =
        \lim_{j\to\infty}
        \frac{\mathcal V_{t_j}^\phi(p)}{t_j}
        =
        \underline\alpha_\phi(p).
\]
This proves \eqref{eq:state-conditioned-occupation-measure-formula}.
\end{proof}

\subsection{Chain-forward support and bounds using basic sets}\label{sec:component-bridge}

Theorem~\ref{thm:state-conditioned-occupation-measure-representation} expresses \(\underline\alpha_\phi(p)\) in terms of \(\mathfrak A_\Gamma(p)\). We now use the supports of these laws to obtain a lower bound in terms of chain-accessible basic sets. Finite exact reachability provides a complementary upper bound. Both bounds can be strict, as shown in Appendix~\ref{app:component-pressure-tests}.

A basic set \(B\) is \emph{chain-accessible from \(p\)} when
\(B\cap\mathcal C_\Gamma(p)\neq\varnothing\); since \(B\) is one chain-equivalence class, this is equivalent to \(B\subseteq\mathcal C_\Gamma(p)\).  For finite exact reachability, write \(p\leadsto_\Gamma q\) if there exist \(r\ge0\) and an exact admissible path
\[
        \xi_0=p,\xi_1,\ldots,\xi_r=q,
        \qquad
        (\xi_i,\xi_{i+1})\in\Gamma\quad(0\le i<r).
\]
A basic set \(B\) is \emph{finitely exactly reachable from \(p\)}
if \(p\leadsto_\Gamma q\) for some \(q\in B\).

\begin{proposition}[Chain-forward support of current-state limits]\label{prop:fixed-initial-chain-forward}
Let \(p\in K_\Gamma^\infty\), let \(\lambda\in\mathfrak A_\Gamma(p)\), and set
\[
m:=(\mathrm{pr}_1)_\#\lambda=(\mathrm{pr}_2)_\#\lambda.
\]
Then
\begin{equation}\label{eq:fixed-initial-chain-forward-support}
\supp m
\subseteq
\mathcal R(\Gamma)\cap\mathcal C_\Gamma(p).
\end{equation}
In particular,
\[
\lambda\!\left(
\bigl(\mathcal R(\Gamma)\cap\mathcal C_\Gamma(p)\bigr)
\times
\bigl(\mathcal R(\Gamma)\cap\mathcal C_\Gamma(p)\bigr)
\right)=1.
\]
\end{proposition}

\begin{proof}
Proposition~\ref{prop:balanced-measures-chain-recurrent} gives \(\supp m\subseteq\mathcal R(\Gamma)\).  We now show that every \(x\in\supp m\) can be reached from \(p\) by an \(\varepsilon\)-chain for every \(\varepsilon>0\), which proves \(\supp m\subseteq\mathcal C_\Gamma(p)\).  Choose \(t_j\to\infty\) and \(\omega^{(j)}\in\Omega_\Gamma(p)\) such that
\(L_{t_j}^{\omega^{(j)}}\Longrightarrow\lambda\).  Their first marginals
\[
m_j:=\frac1{t_j}\sum_{k=0}^{t_j-1}\delta_{\omega_k^{(j)}}
\]
converge weakly to \(m\).

Fix \(x\in\supp m\).  If \(x=p\), then \(x\in\mathcal R(\Gamma)\) already gives \((p,p)\in\mathcal C_\Gamma\).  Suppose \(x\neq p\), and fix \(\varepsilon>0\).  Choose
\[
0<r<\min\{\varepsilon,d_K(p,x)/2\}.
\]
Write
\[
\mathbb B_r(x):=\{y\in K:d_K(x,y)<r\}.
\]
Since \(x\in\supp m\), the open ball \(\mathbb B_r(x)\) has positive \(m\)-measure.  The open-set form of the Portmanteau theorem gives
\[
0<m(\mathbb B_r(x))\le\liminf_{j\to\infty}m_j(\mathbb B_r(x)),
\]
so \(m_j(\mathbb B_r(x))>0\) for all sufficiently large \(j\).  Moreover \(p=\omega_0^{(j)}\notin\mathbb B_r(x)\) by the choice of \(r\).  Hence, for such a \(j\), there is an index \(1\le k<t_j\) with
\(\omega_k^{(j)}\in\mathbb B_r(x)\).  The finite sequence
\[
p=\omega_0^{(j)},\omega_1^{(j)},\ldots,\omega_{k-1}^{(j)},x
\]
is an \(\varepsilon\)-chain.  All preceding steps are exact \(\Gamma\)-edges, and for the final step we take the exact successor \(y=\omega_k^{(j)}\in\Gamma(\omega_{k-1}^{(j)})\); then
\[
d_K(x,y)=d_K(x,\omega_k^{(j)})<r<\varepsilon.
\]
Since \(\varepsilon\) was arbitrary, \(x\in\mathcal C_\Gamma(p)\). The final assertion follows because both marginals of \(\lambda\) assign full mass to \(\mathcal R(\Gamma)\cap\mathcal C_\Gamma(p)\).
\end{proof}

\begin{proposition}[Chain-accessible basic-set lower bound]\label{prop:component-floor-lower-bound}
For every \(p\in K_\Gamma^\infty\),
\begin{equation}\label{eq:component-floor-lower-bound}
\underline\alpha_\phi(p)
\ge
\min_{\substack{B\in\mathcal B(\Gamma)\\ B\cap\mathcal C_\Gamma(p)\neq\varnothing}}
\beta_\phi(B).
\end{equation}
\end{proposition}

\begin{proof}
Fix \(p\in K_\Gamma^\infty\). Choose a minimizing \(\lambda^\ast\in\mathfrak A_\Gamma(p)\) in Theorem~\ref{thm:state-conditioned-occupation-measure-representation}, lift it to \(\widehat\nu^\ast\in\mathcal M_{s_\Gamma}(X_\Gamma)\), and apply \eqref{eq:basic-set-average-decomposition}. For \(\eta\in E\), let \(B_\eta\) be the corresponding basic set and set
\[
\lambda_\eta:=(\operatorname{pr}_{0,1}^{\mathbb Z})_\#\eta,
\qquad
m_\eta:=(\mathrm{pr}_1)_\#\lambda_\eta.
\]
Then \(\lambda_\eta\in\mathcal I(\Gamma_{B_\eta})\), and pushing the decomposition forward gives
\[
\lambda^\ast=\int_E\lambda_\eta\,\Pi(d\eta),
\qquad
m:=(\mathrm{pr}_1)_\#\lambda^\ast
=\int_E m_\eta\,\Pi(d\eta).
\]
By Proposition~\ref{prop:fixed-initial-chain-forward}, \(m(\mathcal C_\Gamma(p))=1\). Hence \(m_\eta(\mathcal C_\Gamma(p))=1\) for \(\Pi\)-almost every \(\eta\). Since \(m_\eta(B_\eta)=1\), it follows that
\(B_\eta\cap\mathcal C_\Gamma(p)\neq\varnothing\)
for \(\Pi\)-almost every \(\eta\).

Thus the family of basic sets in \eqref{eq:component-floor-lower-bound} is nonempty. Let \(\ell\) denote the infimum of \(\beta_\phi(B)\) over this family. For \(\Pi\)-almost every \(\eta\),
\[
\int_{\Gamma_{B_\eta}}T_\phi(y\mid x)\,d\lambda_\eta(x,y)
\ge\beta_\phi(B_\eta)\ge\ell.
\]
Averaging and using the minimizing property of \(\lambda^\ast\) yields \(\underline\alpha_\phi(p)\ge\ell\).

It remains to show that \(\ell\) is attained. Choose basic sets \(B_j\) from this family with \(\beta_\phi(B_j)\to\ell\), and choose \(\lambda_j\in\mathcal I(\Gamma_{B_j})\) attaining \(\beta_\phi(B_j)\). Each \(B_j\) is contained in \(\mathcal C_\Gamma(p)\): this follows from transitivity of \(\mathcal C_\Gamma\), since all points of a basic set are mutually chain-accessible. Regard \(\lambda_j\) as measures in \(\mathcal I(\Gamma)\). By compactness, after passing to a subsequence, \(\lambda_j\Longrightarrow\lambda\in\mathcal I(\Gamma)\). Since \(\mathcal C_\Gamma(p)\) is closed by Lemma~\ref{lem:basic-set-path-measures}, the Portmanteau theorem and continuity of \(T_\phi\) give
\[
\lambda\!\left(\Gamma\cap
\bigl(\mathcal C_\Gamma(p)\times\mathcal C_\Gamma(p)\bigr)\right)=1,
\qquad
\int_\Gamma T_\phi(y\mid x)\,d\lambda(x,y)=\ell.
\]

Apply the same lifting and decomposition argument to \(\lambda\). With a slight abuse of notation, write \(\Pi,E,B_\eta,\lambda_\eta\) again for the objects in this new decomposition. The full-mass condition above again gives \(B_\eta\cap\mathcal C_\Gamma(p)\neq\varnothing\) for \(\Pi\)-almost every \(\eta\). Thus the component costs again satisfy
\[
\int_{\Gamma_{B_\eta}}T_\phi(y\mid x)\,d\lambda_\eta(x,y)
\ge\beta_\phi(B_\eta)\ge\ell
\qquad\text{for \(\Pi\)-almost every }\eta.
\]
Their average equals \(\ell\), so both inequalities are equalities for \(\Pi\)-almost every \(\eta\), as in the proof of Corollary~\ref{cor:basic-set-global-minimizer}. Choosing one such component proves attainment and completes the proof.
\end{proof}

Recall that \(\Omega_{\Gamma_B}\) is the one-sided path space generated by \(\Gamma_B\). For \(q\in B\), write
\[
\Omega_{\Gamma_B}(q)
:=
\{\omega\in\Omega_{\Gamma_B}:\omega_0=q\},
\qquad q\in B.
\]
By Lemma~\ref{lem:basic-set-path-measures}, \(\Gamma_B\) is surjective on \(B\), so \(\Omega_{\Gamma_B}(q)\ne\varnothing\) for every \(q\in B\).  For \(t\ge1\), define the minimum accumulated \(L\)-divergence over the first \(t\) periods among admissible paths starting from \(q\) and remaining in \(B\) by
\begin{equation}\label{eq:basic-set-finite-horizon-value}
\mathcal V_{t,B}^\phi(q)
:=
\min_{\omega\in\Omega_{\Gamma_B}(q)}
\sum_{k=0}^{t-1}T_\phi(\omega_{k+1}\mid\omega_k).
\end{equation}
The minimum is attained because \(\Omega_{\Gamma_B}(q)\) is compact and the displayed cost depends continuously on the first \(t+1\) coordinates.  The restricted lower asymptotic rate is
\begin{equation}\label{eq:basic-set-pointwise-liminf}
\underline\alpha_{\phi,B}(q)
:=
\liminf_{t\to\infty}\frac{\mathcal V_{t,B}^\phi(q)}{t}.
\end{equation}

\begin{corollary}[Bounds for the current-state rate using basic sets]\label{cor:state-conditioned-basic-set-comparisons}
For every \(p\in K_\Gamma^\infty\),
\begin{equation}\label{eq:state-conditioned-basic-set-comparisons}
\min_{\substack{B\in\mathcal B(\Gamma)\\B\cap\mathcal C_\Gamma(p)\neq\varnothing}}
\beta_\phi(B)
\le
\underline\alpha_\phi(p)
\le
\inf_{\substack{B\in\mathcal B(\Gamma),\ q\in B\\p\leadsto_\Gamma q}}
\underline\alpha_{\phi,B}(q).
\end{equation}
The infimum over an empty right-hand family is understood as \(+\infty\).
\end{corollary}

\begin{proof}
The left-hand inequality is Proposition~\ref{prop:component-floor-lower-bound}.  For the right-hand inequality, fix \(B\in\mathcal B(\Gamma)\) and \(q\in B\) with \(p\leadsto_\Gamma q\), and choose an exact path \(p=\xi_0,\ldots,\xi_r=q\).  Concatenating this prefix with a restricted \(t\)-step minimizer in \(B\) gives
\[
\mathcal V_{t+r}^\phi(p)
\le
\sum_{i=0}^{r-1}T_\phi(\xi_{i+1}\mid \xi_i)
+
\mathcal V_{t,B}^\phi(q).
\]
Dividing by \(t+r\) and taking lower limits yields
\(
\underline\alpha_\phi(p)\le\underline\alpha_{\phi,B}(q)
\).
Taking the infimum over all pairs \((B,q)\) with \(p\leadsto_\Gamma q\) gives the right-hand inequality in \eqref{eq:state-conditioned-basic-set-comparisons}.
\end{proof}

\section{Long-Run Growth within a Basic Set}\label{sec:common-basic-set-rates}

The upper bound in \eqref{eq:state-conditioned-basic-set-comparisons} contains the restricted lower asymptotic rate \(\underline\alpha_{\phi,B}(q)\).  We ask when this quantity equals the balanced minimum \(\beta_\phi(B)\), so that the upper bound can be expressed using that minimum.  The conditions below give the stronger conclusion that \(\mathcal V_{t,B}^{\phi}(q)/t\) converges to \(\beta_\phi(B)\).  For every \(t\ge1\), the definitions give
\begin{equation}\label{eq:basic-set-path-minimum}
\begin{aligned}
\min_{x\in B}\mathcal V_{t,B}^{\phi}(x)
&=\min_{x\in B}\min_{\omega\in\Omega_{\Gamma_B}(x)}
\sum_{k=0}^{t-1}T_\phi(\omega_{k+1}\mid\omega_k)\\
&=\min_{\omega\in\Omega_{\Gamma_B}}
\sum_{k=0}^{t-1}T_\phi(\omega_{k+1}\mid\omega_k).
\end{aligned}
\end{equation}
Applying Lemma~\ref{lem:worst-case-birkhoff-averages} to the compact shift system \(\Omega_{\Gamma_B}\), and then Theorem~\ref{thm:transition-measure-form} to the restricted relation \(\Gamma_B\), gives
\begin{equation}\label{eq:basic-set-minimum-limit}
\begin{aligned}
\lim_{t\to\infty}
\min_{x\in B}\frac{\mathcal V_{t,B}^{\phi}(x)}{t}
&=
\min_{\nu\in\mathcal M_\sigma(\Omega_{\Gamma_B})}
\int_{\Omega_{\Gamma_B}} f_\phi\,d\nu \\
&=
\min_{\lambda\in\mathcal I(\Gamma_B)}
\int_{\Gamma_B}T_\phi(y\mid x)\,d\lambda(x,y)
=
\beta_\phi(B).
\end{aligned}
\end{equation}
Equation~\eqref{eq:basic-set-minimum-limit} identifies the asymptotic minimum when the initial state is also allowed to vary over \(B\).  We ask whether \(\mathcal V_{t,B}^{\phi}(q)/t\) converges to \(\beta_\phi(B)\) for a given \(q\in B\).

The next proposition records the finite-path concatenation argument used in the subsequent reachability corollary.  Recall that \(q\leadsto_{\Gamma_B}x\) means that there is a finite exact path in \(\Gamma_B\) from \(q\) to \(x\).

\begin{proposition}[Finite-prefix criterion for the basic-set rate]\label{prop:finite-reachability-basic-set-minimum}
Let \(B\in\mathcal B(\Gamma)\) and \(q\in B\).  Suppose that there exist \(x\in B\) and \(\omega\in\Omega_{\Gamma_B}(x)\) such that
\[
q\leadsto_{\Gamma_B}x
\]
and
\[
\lim_{t\to\infty}
\frac1t\sum_{k=0}^{t-1}
T_\phi(\omega_{k+1}\mid\omega_k)
=
\beta_\phi(B).
\]
Then
\begin{equation}\label{eq:full-basic-set-rate}
\underline\alpha_{\phi,B}(q)
=\lim_{t\to\infty}\frac{\mathcal V_{t,B}^{\phi}(q)}{t}
=\beta_\phi(B).
\end{equation}
\end{proposition}

\begin{proof}
Choose a finite exact \(\Gamma_B\)-admissible path
\[
q=\xi_0,\xi_1,\ldots,\xi_r=x.
\]
For \(t>r\), concatenate this prefix with the first \(t-r\) transitions of \(\omega\).  The resulting \(t\)-step path is admissible and starts from \(q\), so
\[
\mathcal V_{t,B}^{\phi}(q)
\le
\sum_{i=0}^{r-1}T_\phi(\xi_{i+1}\mid \xi_i)
+
\sum_{k=0}^{t-r-1}T_\phi(\omega_{k+1}\mid\omega_k).
\]
The prefix cost is independent of \(t\).  After division by \(t\), the second sum is \((t-r)/t\) times its \((t-r)\)-step average.  Since \((t-r)/t\to1\), the assumed long-run average gives
\[
\limsup_{t\to\infty}
\frac{\mathcal V_{t,B}^{\phi}(q)}{t}
\le
\beta_\phi(B).
\]
On the other hand,
\[
\min_{y\in B}\mathcal V_{t,B}^{\phi}(y)
\le
\mathcal V_{t,B}^{\phi}(q)
\]
for every \(t\).  Equation~\eqref{eq:basic-set-minimum-limit} therefore yields
\[
\beta_\phi(B)
\le
\liminf_{t\to\infty}
\frac{\mathcal V_{t,B}^{\phi}(q)}{t}.
\]
The two bounds prove \eqref{eq:full-basic-set-rate}.
\end{proof}

\begin{remark}[Finite-horizon and single-path values]\label{rem:finite-horizon-single-path-values}
Under the hypotheses of Proposition~\ref{prop:finite-reachability-basic-set-minimum}, the finite-horizon and single-path values agree for paths starting from \(q\) and remaining in \(B\):
\[
\begin{aligned}
&\lim_{t\to\infty}
\min_{\eta\in\Omega_{\Gamma_B}(q)}
\frac1t\sum_{k=0}^{t-1}T_\phi(\eta_{k+1}\mid\eta_k)\\
&\quad=
\min_{\eta\in\Omega_{\Gamma_B}(q)}
\liminf_{t\to\infty}\frac1t\sum_{k=0}^{t-1}T_\phi(\eta_{k+1}\mid\eta_k)\\
&\quad=\beta_\phi(B).
\end{aligned}
\]
By Proposition~\ref{prop:finite-reachability-basic-set-minimum}, the first term equals \(\beta_\phi(B)\). For every \(\eta\in\Omega_{\Gamma_B}(q)\) and \(t\ge1\),
\[
\frac{\mathcal V_{t,B}^{\phi}(q)}{t}
\le
\frac1t\sum_{k=0}^{t-1}T_\phi(\eta_{k+1}\mid\eta_k).
\]
Taking lower limits therefore shows that every single-path value is at least \(\beta_\phi(B)\). The concatenated path used in the proof of that proposition has an average converging to \(\beta_\phi(B)\), so it attains this lower bound. This proves the displayed identity.
\end{remark}

Pairwise finite exact reachability on \(B\) yields this convergence for every initial state in \(B\).

\begin{corollary}[Pairwise finite exact reachability on a basic set]\label{cor:basic-set-pairwise-reachability}
Let \(B\in\mathcal B(\Gamma)\). Suppose that
\[
u\leadsto_{\Gamma_B}v
\qquad\text{for every }u,v\in B.
\]
Then
\[
\underline\alpha_{\phi,B}(q)
=\lim_{t\to\infty}\frac{\mathcal V_{t,B}^{\phi}(q)}{t}
=\beta_\phi(B),
\qquad q\in B.
\]
\end{corollary}

\begin{proof}
By Theorem~\ref{thm:invariant-L-divergence-representation} applied to \(\Omega_{\Gamma_B}\) and Theorem~\ref{thm:transition-measure-form} applied to \(\Gamma_B\),
\[
\min_{\nu\in\mathcal M_\sigma(\Omega_{\Gamma_B})}
\int_{\Omega_{\Gamma_B}}f_\phi\,d\nu
=\beta_\phi(B).
\]
The functional \(\nu\mapsto\int f_\phi\,d\nu\) is continuous and affine on the compact convex set \(\mathcal M_\sigma(\Omega_{\Gamma_B})\). Its minimizing set is a nonempty compact face. By the Krein--Milman theorem, this face contains an extreme point, which is also extreme in \(\mathcal M_\sigma(\Omega_{\Gamma_B})\). By \cite[Theorem~6.10(iii)]{Walters1982ErgodicTheory}, the extreme invariant measures are precisely the ergodic ones. Choose an ergodic minimizing measure \(\nu\). Birkhoff's ergodic theorem \cite[Theorem~1.14 and the following remark]{Walters1982ErgodicTheory} gives, for \(\nu\)-almost every \(\omega\in\Omega_{\Gamma_B}\),
\[
\frac1t\sum_{k=0}^{t-1}T_\phi(\omega_{k+1}\mid\omega_k)
=\frac1t\sum_{k=0}^{t-1}f_\phi(\sigma^k\omega)
\longrightarrow
\int_{\Omega_{\Gamma_B}}f_\phi\,d\nu
=\beta_\phi(B).
\]
Choose one such path and set \(x=\omega_0\). Then \(\omega\in\Omega_{\Gamma_B}(x)\), and the assumption gives \(q\leadsto_{\Gamma_B}x\) for every \(q\in B\). Proposition~\ref{prop:finite-reachability-basic-set-minimum} applies to each \(q\).
\end{proof}

\section{Discussion and Outlook}\label{sec:discussion}

This paper evaluates worst-case relative growth for a fixed functionally
generated portfolio under prescribed market-path constraints. It distinguishes
the global rate \(\gamma(\Phi,\Omega)\) from the lower asymptotic rate
\(\underline\alpha_\phi(p)\) associated with a specified current state in
the transition-relation setting. The former admits an invariant-measure
characterization; the latter depends on the empirical transition laws
generated from that state. Chain reachability and finite exact reachability
give complementary basic-set bounds, both of which can be strict
(Appendix~\ref{app:component-pressure-tests}). Within a basic set,
Section~\ref{sec:common-basic-set-rates} gives sufficient conditions for
\(\mathcal V_{t,B}^{\phi}(q)/t\) to converge to \(\beta_\phi(B)\).

One direction for future work is to equip the transition relation with
transition probabilities and study expected or almost-sure relative growth.
Such a model could describe both changes in market weights and periods
during which they remain unchanged.

A second direction is to study path constraints that depend on past market
weights. The general path space \(\Omega\) already allows such constraints,
but their implications for empirical transition laws require further study.
The signature methods used in the portfolio constructions of
\cite{CuchieroMoeller2025Signature} offer one way to describe past market
movements. Future work could investigate whether these methods can help
formulate path constraints and analyze the resulting worst-case relative
growth rate.

Finally, one may optimize \(\gamma(\Phi,\Omega)\) over a prescribed class
of generating functions while keeping \(\Omega\) fixed. The measure
representations established here provide a starting point for studying
the existence of an optimal generator and developing computational methods
for this problem.

\appendix
\section{Strictness of the Lower and Upper Bounds}\label{app:component-pressure-tests}

The two examples below show that each inequality in \eqref{eq:state-conditioned-basic-set-comparisons} can be strict.  The first has a unique chain-accessible basic set and a zero global rate, but a strictly positive current-state rate.  The second has a zero current-state rate while the infimum in the upper bound is strictly positive.

Work on the compact interval \(K=\{(x,1-x):x\in[a,b]\}\subset\Delta_2\), where \(0<a<b<1\).  Fix a positive \(C^1\), strictly concave generator \(\Phi\) on a neighborhood of \(K\), and set \(\phi=\log\Phi\).  Since
\[
T_\phi(q\mid p)
=
\log\frac{\Phi(p)+D\Phi(p)[q-p]}{\Phi(q)},
\]
strict concavity gives
\begin{equation}\label{eq:strict-L-positive-off-diagonal}
T_\phi(q\mid p)>0
\qquad(p,q\in K,\ p\neq q),
\end{equation}
while \(T_\phi(p\mid p)=0\).

\begin{example}[A chain-accessible basic set with a strict current-state gap]\label{ex:strict-component-lower}
We construct a relation with a unique basic set \(B\), chain-accessible from \(z_0\), for which
\[
\mathfrak A_\Gamma(z_0)\subsetneq\mathcal I(\Gamma_B),
\qquad
\gamma(\Phi,\Omega_\Gamma)=0<\underline\alpha_\phi(z_0).
\]

Choose distinct states \(a_0,a_1\in K\) and a sequence of pairwise distinct states \((z_n)_{n\ge0}\subset K\setminus\{a_0,a_1\}\) such that
\[
z_{2n}\to a_0,
\qquad
z_{2n+1}\to a_1.
\]
Let \(B:=\{a_0,a_1\}\) and
\[
\Gamma:=B\times B
\;\cup\;
\{(z_n,z_{n+1}):n\ge0\}.
\]
The only limit points of the sequence \((z_n,z_{n+1})_{n\ge0}\) are \((a_0,a_1)\) and \((a_1,a_0)\), both of which belong to \(B\times B\). Hence \(\Gamma\) is compact.

We claim that \(\mathcal R(\Gamma)=B\). Every point of \(B\) is chain recurrent, and points outside \(B\cup\{z_n:n\ge0\}\) have no successors. It remains to show that no \(z_n\) is chain recurrent. Fix \(n\ge0\), and set
\[
E_n:=\{z_0,\ldots,z_n\},
\qquad
F_n:=B\cup\{z_m:m\ge n+1\}.
\]
These are disjoint compact sets. Choose \(0<\varepsilon<d_K(E_n,F_n)\). The first exact successor of \(z_n\) is \(z_{n+1}\in F_n\). Every nonterminal point of an \(\varepsilon\)-chain has a successor and hence belongs to \(E_n\cup F_n\). Since all exact successors of points of \(F_n\) remain in \(F_n\), the separation of the two sets forces every nonterminal point after \(z_n\) to lie in \(F_n\) and prevents the final point from lying in \(E_n\). Thus no \(\varepsilon\)-chain from \(z_n\) can return to \(z_n\), proving the claim. Since \(B\times B\subset\Gamma\), the set \(B\) is the unique relation basic set.

From \(z_0\) the admissible path is unique:
\[
\omega=(z_0,z_1,z_2,\ldots).
\]
The even and odd transition subsequences converge respectively to \((a_0,a_1)\) and \((a_1,a_0)\), each with asymptotic density \(1/2\).  Hence, with
\[
\lambda_{\mathrm{alt}}
:=
\frac12\delta_{(a_0,a_1)}
+
\frac12\delta_{(a_1,a_0)},
\]
one has
\[
L_t^\omega\Longrightarrow\lambda_{\mathrm{alt}},
\qquad
\mathfrak A_\Gamma(z_0)=\{\lambda_{\mathrm{alt}}\}.
\]
By Lemma~\ref{lem:accessible-transition-measures-stationary}, \(\lambda_{\mathrm{alt}}\) is balanced. Its explicit form shows that it is supported on \(B\times B\), so \(\lambda_{\mathrm{alt}}\in\mathcal I(\Gamma_B)\).  Theorem~\ref{thm:state-conditioned-occupation-measure-representation} gives
\[
\underline\alpha_\phi(z_0)
=
\int_{\Gamma_B}T_\phi(y\mid x)\,d\lambda_{\mathrm{alt}}(x,y)
=
\frac12T_\phi(a_1\mid a_0)
+
\frac12T_\phi(a_0\mid a_1)
>0.
\]
On the other hand,
\[
\delta_{(a_0,a_0)}\in
\mathcal I(\Gamma_B)\setminus\mathfrak A_\Gamma(z_0).
\]
More explicitly,
\[
\mathfrak A_\Gamma(z_0)
=
\{\lambda_{\mathrm{alt}}\}
\subsetneq
\mathcal I(\Gamma_B).
\]
In particular, \(\beta_\phi(B)=0\).  The basic set \(B\) is chain-accessible from \(z_0\): follow a sufficiently long prefix of the path \((z_n)_{n\ge0}\) and use the convergence \(z_{2N+1}\to a_1\) for the final \(\varepsilon\)-step.  It is not finitely exactly reachable from \(z_0\), because the unique exact path from \(z_0\) remains among the states \(z_n\).

Since \(B\) is the unique basic set, \eqref{eq:global-basic-set-minimum} gives
\[
\gamma(\Phi,\Omega_\Gamma)
=
\beta_\phi(B)
=
0
<
\underline\alpha_\phi(z_0).
\]
Thus the lower bound in \eqref{eq:state-conditioned-basic-set-comparisons} is strict: minimizing over \(\mathcal I(\Gamma_B)\) allows zero-cost self-loops, whereas the only empirical transition law generated from \(z_0\) is \(\lambda_{\mathrm{alt}}\).

For a concrete equal-weight example, take
\[
K=\{(x,1-x):1/4\le x\le3/4\},\qquad
\Phi(p)=\sqrt{p_1p_2},
\]
and choose \(a_0=(1/3,2/3)\) and \(a_1=(2/3,1/3)\).
The generated portfolio is \(\pi=(1/2,1/2)\). Since the one-step
relative wealth factor in either direction is \(5/4\) and
\(\Phi(a_0)=\Phi(a_1)\),
\[
T_\phi(a_1\mid a_0)=T_\phi(a_0\mid a_1)=\log(5/4).
\]
Thus \(\gamma(\Phi,\Omega_\Gamma)=0<\underline\alpha_\phi(z_0)=\log(5/4)\). Even with a unique basic set, the global rate need not equal the lower asymptotic rate from a specified initial state.
\end{example}

\begin{example}[Strictness of the upper bound]\label{ex:strict-component-upper}
We construct a relation for which the current-state rate from \(a_0\) is zero, but the infimum in the upper bound of \eqref{eq:state-conditioned-basic-set-comparisons} is positive.  A path can approach a zero-cost basic set without reaching it in finitely many exact transitions.

Choose distinct states \(a_0,a_1,b_0\in K\).  Set
\[
A:=\{a_0,a_1\},
\qquad
B:=\{b_0\}.
\]
Choose pairwise distinct states
\[
z_n\in K\setminus(A\cup B),
\qquad z_n\to b_0,
\]
and set
\[
\Gamma
:=
\{(a_0,a_1),(a_1,a_0),(a_0,z_0)\}
\cup
\{(z_n,z_{n+1}):n\ge0\}
\cup
\{(b_0,b_0)\}.
\]
The relation is compact because \((z_n,z_{n+1})\to(b_0,b_0)\).

We claim that \(\mathcal R(\Gamma)=A\cup B\). The exact two-cycle makes every point of \(A\) chain recurrent, and the self-loop makes \(b_0\) chain recurrent. Points outside \(A\cup B\cup\{z_n:n\ge0\}\) have no successors. To exclude \(z_n\), fix \(n\ge0\), and set
\[
E_n:=\{z_0,\ldots,z_n\},
\qquad
F_n:=B\cup\{z_m:m\ge n+1\}.
\]
Choose
\[
0<\varepsilon<\min\{d_K(E_n,F_n),d_K(A,F_n)\}.
\]
As in Example~\ref{ex:strict-component-lower}, every nonterminal point after \(z_n\) in an \(\varepsilon\)-chain starting at \(z_n\) must lie in \(F_n\): it must have a successor, and the chosen separation excludes both \(E_n\) and \(A\). The final point cannot be \(z_n\in E_n\), so \(z_n\) is not chain recurrent. This proves the claim. Moreover, \(A\) and \(B\) are distinct basic sets: if
\[
Z:=B\cup\{z_n:n\ge0\},
\]
then \(d_K(A,Z)>0\), and for \(\varepsilon<d_K(A,Z)\) no \(\varepsilon\)-chain from \(b_0\) can reach \(A\).

Take \(p=a_0\).  From either state \(q\in A\), a path restricted to \(A\) alternates between \(a_0\) and \(a_1\), so
\[
\underline\alpha_{\phi,A}(q)
=\lim_{t\to\infty}\frac{\mathcal V_{t,A}^\phi(q)}{t}
=\frac12T_\phi(a_1\mid a_0)+\frac12T_\phi(a_0\mid a_1)>0,
\qquad q\in A.
\]

The admissible exit path
\[
a_0\to z_0\to z_1\to z_2\to\cdots
\]
has empirical transition measures converging to \(\delta_{(b_0,b_0)}\), because \((z_n,z_{n+1})\to(b_0,b_0)\) and the initial edge has vanishing empirical weight.  Since \(T_\phi\ge0\), continuity of \(T_\phi\) and \(T_\phi(b_0\mid b_0)=0\) imply
\[
\underline\alpha_\phi(a_0)=0.
\]
Moreover, \(z_n\to b_0\), so \(B\) is chain-accessible from \(a_0\).  It is not finitely exactly reachable from \(a_0\): after taking the exit edge, every exact path remains among the states \(z_n\).  Thus \(A\) is the only basic set finitely exactly reachable from \(a_0\), and both its states are reachable.  Consequently, the infimum in \eqref{eq:state-conditioned-basic-set-comparisons} equals
\[
\min_{q\in A}\underline\alpha_{\phi,A}(q)
=\frac12T_\phi(a_1\mid a_0)+\frac12T_\phi(a_0\mid a_1)
>\underline\alpha_\phi(a_0)=0.
\]
This proves strictness of the upper bound.

Using the same \(K\), equal-weight generator, and two-cycle states
as in Example~\ref{ex:strict-component-lower}, choose \(b_0=(1/2,1/2)\).
Then
\[
\underline\alpha_{\phi,A}(a_0)
=\underline\alpha_{\phi,A}(a_1)
=\log(5/4),
\qquad
\underline\alpha_\phi(a_0)=0.
\]
Thus the lower asymptotic rate \(\underline\alpha_\phi(a_0)\) is strictly smaller than the restricted lower asymptotic rate \(\underline\alpha_{\phi,A}(a_0)\). The exit path realizes the value zero by approaching \(b_0\) without reaching it in finitely many transitions.
\end{example}

\section{A Three-Asset Restricted Path Space with Positive Worst-Case Relative Growth}\label{app:three-asset-positive-growth}

The relation \(\Gamma_\varepsilon\) below bounds one-step changes but permits constant paths, so \(\gamma(\Phi,\Omega_{\Gamma_\varepsilon})=0\). We impose an additional lower bound on accumulated squared increments to define \(\Omega_{\rho,C}\subset\Omega_{\Gamma_\varepsilon}\). For the equal-weighted portfolio, this condition yields a positive global rate whenever \(\Omega_{\rho,C}\ne\varnothing\).

Let \(n=3\), let \(K\subset\Delta_3\) be compact and convex with nonempty interior relative to the affine hull of \(\Delta_3\), and fix \(\varepsilon>0\).  Define
\begin{equation}\label{eq:appB-local-relation}
\Gamma_\varepsilon
:=
\{(p,q)\in K\times K:\|q-p\|\le\varepsilon\},
\end{equation}
where \(\|\cdot\|\) denotes the Euclidean norm on \(\mathbb R^3\).  By convexity, \(q\leadsto_{\Gamma_\varepsilon}x\) for every \(q,x\in K\).

In this appendix, fix the equal-weighted generator
\[
\Phi(p):=(p_1p_2p_3)^{1/3},
\qquad
\phi:=\log\Phi,
\]
and let \(\pi=(1/3,1/3,1/3)\) be its associated portfolio.  For \(p,q\in K\), set temporarily
\[
r_i:=\frac{q_i}{p_i},
\qquad
\bar r:=\frac13\sum_{i=1}^3r_i.
\]
Then
\begin{equation}\label{eq:appB-ew-ratio-formula}
T_\phi(q\mid p)
=
\log \bar r
-
\frac13\sum_{i=1}^3\log r_i.
\end{equation}
Indeed,
\[
1+\langle\nabla\phi(p),q-p\rangle
=
\frac13\sum_{i=1}^3\frac{q_i}{p_i},
\qquad
\phi(q)-\phi(p)
=
\frac13\sum_{i=1}^3\log\frac{q_i}{p_i}.
\]

There is a constant \(c_K>0\) such that
\begin{equation}\label{eq:appB-ew-coercivity}
T_\phi(q\mid p)
\ge
c_K\|q-p\|^2,
\qquad p,q\in K.
\end{equation}
Since \(K\) is a compact subset of \(\Delta_3\), there is a finite constant \(M\), depending only on \(K\), such that \(r_i\le M\) for all \(p,q\in K\) and \(i=1,2,3\).  Since
\[
(\log x)''=-\frac1{x^2}\le-\frac1{M^2}
\qquad (0<x\le M),
\]
the second-order upper bound at \(\bar r\) gives, for each \(i\),
\[
\log r_i
\le
\log\bar r
+\frac{r_i-\bar r}{\bar r}
-\frac{(r_i-\bar r)^2}{2M^2}.
\]
Averaging over \(i\) cancels the linear term because \(\sum_i(r_i-\bar r)=0\).  Hence \eqref{eq:appB-ew-ratio-formula} yields
\begin{equation}\label{eq:appB-ratio-variance}
T_\phi(q\mid p)
\ge
\frac1{6M^2}
\sum_{i=1}^3(r_i-\bar r)^2.
\end{equation}
Because \(\sum_i p_ir_i=\sum_iq_i=1\),
\[
1-\bar r
=
\sum_{i=1}^3p_i(r_i-\bar r),
\]
and hence, by Cauchy--Schwarz,
\[
(1-\bar r)^2
\le
\sum_{i=1}^3(r_i-\bar r)^2.
\]
Also \(q_i-p_i=p_i(r_i-1)\) and \(p_i\le1\), so
\[
\|q-p\|^2\le\sum_{i=1}^3(r_i-1)^2
=\sum_{i=1}^3(r_i-\bar r)^2+3(\bar r-1)^2
\le4\sum_{i=1}^3(r_i-\bar r)^2.
\]
Together with \eqref{eq:appB-ratio-variance}, this proves \eqref{eq:appB-ew-coercivity}, for example with \(c_K=1/(24M^2)\).

For \(\rho>0\) and \(C\ge0\), let \(\Omega_{\rho,C}\subset\Omega_{\Gamma_\varepsilon}\) consist of paths satisfying
\begin{equation}\label{eq:appB-euclidean-increment-bound}
\sum_{k=s}^{s+t-1}
\|\omega_{k+1}-\omega_k\|^2
\ge
\rho t-C,
\qquad s\ge0,\quad t\ge1.
\end{equation}
The requirement for every starting time \(s\) ensures that \(\Omega_{\rho,C}\) is forward shift-invariant.

\begin{proposition}[Positive worst-case relative growth on \(\Omega_{\rho,C}\)]\label{prop:appB-positive-growth}
Suppose that \(\Omega_{\rho,C}\ne\varnothing\). Then \(\Omega_{\rho,C}\) is compact and forward shift-invariant. Moreover, for every \(t\ge1\),
\begin{equation}\label{eq:appB-finite-horizon-growth}
\min_{\omega\in\Omega_{\rho,C}}\log V_\pi(t,\omega)
\ge c_K(\rho t-C)-\bigl(\max_K\phi-\min_K\phi\bigr).
\end{equation}
Consequently,
\begin{equation}\label{eq:appB-positive-gamma}
\gamma(\Phi,\Omega_{\rho,C})\ge c_K\rho>0.
\end{equation}
\end{proposition}

\begin{proof}
The relation \(\Gamma_\varepsilon\) is compact, so \(\Omega_{\Gamma_\varepsilon}\) is compact. For each fixed \(s\ge0\) and \(t\ge1\), the left-hand side of \eqref{eq:appB-euclidean-increment-bound} depends continuously on finitely many coordinates. Hence \(\Omega_{\rho,C}\), as an intersection of closed subsets of \(\Omega_{\Gamma_\varepsilon}\), is compact. Forward shift-invariance follows from the requirement that \eqref{eq:appB-euclidean-increment-bound} hold for every \(s\ge0\).

For every \(\omega\in\Omega_{\rho,C}\), \eqref{eq:appB-ew-coercivity} and \eqref{eq:appB-euclidean-increment-bound} give
\[
\sum_{k=0}^{t-1}T_\phi(\omega_{k+1}\mid\omega_k)\ge c_K(\rho t-C).
\]
The endpoint term in the pathwise wealth decomposition satisfies
\[
\phi(\omega_t)-\phi(\omega_0)\ge\min_K\phi-\max_K\phi.
\]
Combining these bounds and taking the minimum over \(\Omega_{\rho,C}\) gives \eqref{eq:appB-finite-horizon-growth}. By Theorem~\ref{thm:invariant-L-divergence-representation}, the limit defining \(\gamma(\Phi,\Omega_{\rho,C})\) exists. Dividing by \(t\) and letting \(t\to\infty\) yields \eqref{eq:appB-positive-gamma}.
\end{proof}

For nonemptiness, choose distinct \(p^-,p^+\in K\) with \(\|p^+-p^-\|\le\varepsilon\).  The alternating path
\[
p^-,p^+,p^-,p^+,\ldots
\]
belongs to \(\Omega_{\rho,0}\) with
\[
\rho=\|p^+-p^-\|^2>0.
\]

\begingroup\small
\bibliographystyle{plain}
\bibliography{references}
\endgroup

\end{document}